\documentclass[reqno,12pt]{amsart}

\usepackage{graphicx}
\usepackage{amssymb}
\usepackage{amsmath}
\usepackage[hidelinks]{hyperref}
\usepackage{microtype}
\usepackage[nameinlink]{cleveref}

\date{}

\theoremstyle{plain}
\newtheorem{theorem}{Theorem}

\newtheorem{proposition}[theorem]{Proposition}

\newtheorem{rem}{Remark}

\theoremstyle{definition}

\theoremstyle{remark}

\def\N{{\mathbb N}}

\DeclareMathOperator{\lk}{lk}

\title{\'Etude on the Delta-unknotting number}

\author{Sebastian Baader and Lukas Lewark}

\hypersetup{pdfauthor={\authors},pdftitle={\shorttitle}}

\begin{document}

\begin{abstract} We derive an almost sharp lower bound on the Casson invariant of positive braid knots. As a consequence, we show that the $\Delta$\nobreakdash-unknotting number is quasi-additive on the set of positive and negative braids on a fixed number of strands. 
\end{abstract}

%\dedicatory{}

\maketitle

\section{Introduction}

The recent discovery of the non-additivity of the unknotting number $u(K)$ of knots by Brittenham and Hermiller~\cite{BH} leaves us with many interesting questions, thought to be superseded by the (hypothetical) additivity so far. Is the unknotting number quasi-additive, similar to the minimal crossing number~\cite{L}, i.e.~does there exist a constant $c>0$ with $u(K_1 \# K_2) \geq c(u(K_1)+u(K_2))$ for all knots $K_1,K_2$? Does $u(K^n) \geq n$ hold for all knots $K$ and for all $n \in \N$?
What are natural families of knots on which the unknotting number is additive?
In this paper we answer none of these questions. Instead, we consider another unknotting operation for knots, called $\Delta$-move or clasp-pass move, shown in \Cref{delta}.

\begin{figure}[htb]
\includegraphics[scale=1.0]{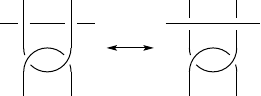}
\caption{$\Delta$-move}
\label{delta}
\end{figure}

The fact that the $\Delta$-move is an unknotting operation for knots was first proved in~\cite{MN}, and can be seen as a special case of Habiro's famous theorem on Vassiliev invariants~\cite{H}. The corresponding unknotting number, denoted by $u_\Delta(K)$, admits a powerful lower bound by the (absolute value of the) Casson invariant $a_2(K)$, i.e.~by the second coefficient of the Conway polynomial of knots, written as
$$\nabla_K(z)=1+a_2(K)z^2+a_4(K)z^4+\ldots+a_{2n}(K)z^{2n}.$$

A remarkable manifestation of the strength of the bound $|a_2(K)| \leq u_\Delta(K)$, proved by Okada~\cite{O}, is the fact that it implies the additivity of $u_\Delta$ on all knots with up to seven crossings.
This makes it potentially hard to come up with an example disproving the additivity of $u_\Delta$. As in the case of the unknotting number, the quasi-additivity of the $\Delta$\nobreakdash-unknotting number remains open. It is under the light of this that our main result is to be savoured.

Let $B_m^+$ be the set of positive braids on $m$ strands, i.e.~the set of finite products of the positive standard generators $\sigma_1,\ldots,\sigma_{m-1} \in B_m$, and similarly let $B_m^-$ be the set of negative braids on $m$ strands. We denote by $\widehat{B_m^+}$ resp.~$\widehat{B_m^-}$ the sets of links that are obtained as closures of braids in $B_m^+$ resp.~$B_m^-$.

\begin{theorem}
For all knots $K_1,K_2 \in \widehat{B_m^+} \cup \widehat{B_m^-}$:
$$u_\Delta(K_1 \# K_2) \geq \frac{1}{48m} \left(u_\Delta(K_1)+u_\Delta(K_2)\right).$$
\label{quasi}
\end{theorem}

In simple words, the $\Delta$\nobreakdash-unknotting number is quasi-additive on the family of knots in $\widehat{B_m^+} \cup \widehat{B_m^-}$. A stronger statement is true in the special case $m=3$. Indeed, for knots $K$ that are closures of positive $3$-braids, the $\Delta$\nobreakdash-unknotting number $u_\Delta(K)$ coincides with the Casson invariant~$a_2(K)$~\cite{NNU}. This implies the  additivity of $u_\Delta$ on $\widehat{B_3^+} \cup \widehat{B_3^-}$, since the Casson invariant is additive and invariant under the mirror image of knots.

We will derive \Cref{quasi} from the following proposition strengthening a special case of a result by Van Buskirk, which provides positive lower and upper bounds on all coefficients of the Conway polynomial of positive braid links~\cite{VB}. The second inequality in the statement below is in fact taken from that reference. We denote the minimal Seifert genus of a knot $K$ by $g(K)$.

\begin{proposition}
Let $K$ be a knot represented as the closure of a positive $m$-braid, and set $g=g(K)$. Then the Casson invariant $a_2(K)$ satisfies
$$\frac{g^2}{6m} \leq a_2(K) \leq \frac{g(g+1)}{2}.$$
\label{casson}
\end{proposition}

Similar lower bounds were recently obtained for the high degree coefficients of the HOMFLY polynomial of positive braid knots~\cite{I}. Our bound complements these, and also improves on Stoimenow's bound in~\cite{St} (Theorem~8.2) for $m$ large enough. The upper bound is sharp, as can be seen by considering the family of torus knot $T(2,2n+1)$ with genus $g=n$ and Casson invariant $a_2(T(2,2n+1))=\frac{n(n+1)}{2}$. The lower bound is almost sharp, as shows the following family of examples. Let $3_1^n$ denote the $n$-fold connected sum of positive trefoil knots. We may represent the knot $3_1^n$ as the closure of the braid $\sigma_1^3 \sigma_2^3 \cdots \sigma_n^3 \in {B_{n+1}^+}$; we observe $g(3_1^n)=n$. The lower bound $\frac{n^2}{6(n+1)}$ converges to $\frac{n}{6}$, as $n$ tends to infinity, a factor six off of the actual value of the Casson invariant: $a_2(3_1^n)=n$.
Nevertheless, it is conceivable that the constant $\frac{1}{48m}$ in the bound of \Cref{quasi} can be strengthened to a constant not depending on $m$, thus proving quasi-additivity of $u_\Delta$ on the set of all positive and negative braid knots.

We will prove both \Cref{casson} and \Cref{quasi} in the next section.

\section{Estimating the Casson invariant}

The proof of \Cref{casson} makes use of the well-known skein relation for the Casson invariant of knots~\cite{PV}, of which we state a braided version here.
Let $\alpha \in B_m^+$ be a positive braid whose closure is a knot, and let $i \leq m-1$ be any positive index. Set $K_-$, $K_+$, $L$ to be the closures of the braids $\alpha$, $\alpha \sigma_i^2$, $\alpha \sigma_i \in B_m^+$, respectively. Then 
$$a_2(K_+)=a_2(K_-)+\lk(L),$$
where $\lk(L)$ denotes the linking number of the two-component link $L$, endowed with the natural orientation induced by its defining braid $\alpha \sigma_i$.

\begin{proof}[Proof of \Cref{casson}]
Let $K$ be the closure of a positive braid $\beta \in B_m^+$ with $n$ crossings, with minimal Seifert genus $g=g(K)$. Bennequin's celebrated genus formula~\cite{B} implies
$$g(K)=\frac{1}{2}(n-m+1).$$  
We start by recapitulating that the second inequality, $a_2(K) \leq \frac{g(g+1)}{2}$, was already derived in~\cite{VB}. We prove the first inequality, $\frac{g^2}{6m} \leq a_2(K)$, for all $m$ simultaneously by induction on $g$, the case $g=0$ being trivial.
Suppose $g \geq 1$, so $K$ is a non-trivial knot. It is well-known that the positive braid $\beta \in B_m^+$ representing $K$ can be chosen to end by the square of a generator, see for example~\cite{R}: there exists $\alpha \in B_m^+$ and an index $i \leq m-1$ with $\beta=\alpha \sigma_i^2$. By the above skein relation, with $K=K_+$, we have $a_2(K)=a_2(K_+)=a_2(K_-)+\lk(L)$. We distinguish two cases:

\medskip
\textbf{Case 1.} $\lk(L) \geq \frac{g}{3m}$.

\noindent We use the induction hypothesis for the knot $K_-$, which satisfies $g(K_-)=g-1$, again by Bennequin's formula, and obtain:
$$a_2(K) \geq \frac{(g-1)^2}{6m}+\frac{g}{3m} %= \frac{(g-1)^2+2g}{4m}
\geq \frac{g^2}{6m}.$$

\medskip
\begin{figure}
\includegraphics{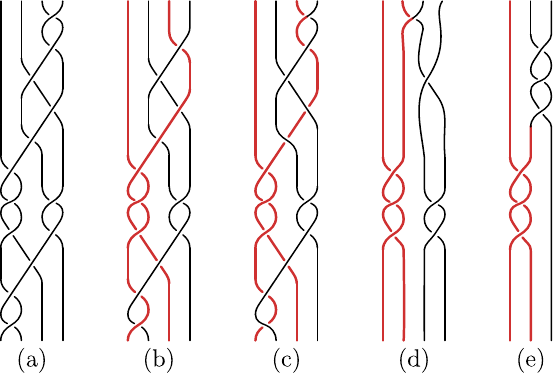}
\caption{Illustration of the proof of \Cref{casson}.
(a)~A positive braid $\beta = \alpha\sigma_3^2$ with $\alpha,\beta\in B_m^+$ for $m = 4$.
(b)~The link $L$ with $\lk(L) = 3$ obtained as closure of~$\alpha \sigma_3$.
(c),~(d)~Three crossing changes turn the closure $K$ of $\beta$ into a sum $K_1\# K_2$.
(e)~$K_1 \# K_2$ is the closure of a positive braid on $m-1 = 3$ strands.}
\label{fig:example}
\end{figure}
\textbf{Case 2.} $\lk(L) < \frac{g}{3m}$.

\noindent The link $L$ being the closure of the positive braid $\alpha \sigma_i \in B_m^+$, we conclude that the two components of $L$ can be split by changing $\lk(L) < \frac{g}{3m}$ crossings in the braid $\alpha \sigma_i$.
Indeed, by changing precisely one half of the crossings between the two components of $L$,
one component can be put on top of the other.
%, via a sequence of closures of positive braid links.
As a consequence, if we change the same crossings in the braid $\alpha \sigma_i^2$ representing the knot $K$, we obtain %a sequence of positive braid knots that terminates with the
a connected sum of two knots $K_1,K_2$, both of which are closures of positive braids $\beta_1 \in B_{m_1}^+$, $\beta_2 \in B_{m_2}^+$, with $m_1+m_2 = m$.
This is illustrated in \Cref{fig:example}.
Again, Bennequin's formula implies
$$g(K_1 \# K_2) \geq g-\frac{g}{3m},$$
since we replaced at most $\frac{g}{3m}$ positive crossings in the braid $\alpha \sigma_i^2$ to negative ones. Now comes the crucial observation: the connected sum $K_1 \# K_2$ can be represented by a positive braid in $B_{m-1}^+$, i.e.~by a positive braid with one strand less than the original one!
See again \Cref{fig:example} for an example.
%This is very simple; just think of the connected sum of two trefoil knots, represented by the braid $\sigma_1^3 \sigma_2^3 \in B_3^+$.
In the general case, we take the product of the braid $\beta_1$ and a version of the braid $\beta_2$ with the index of all generators shifted by $m_1-1$; this results in a positive braid with $m_1+m_2-1=m-1$ strands.

We use the induction hypothesis for the knot $K_1 \# K_2$ and obtain
$$%a_2(K) \geq
a_2(K_1 \# K_2) \geq \frac{(g-\frac{g}{3m})^2}{6(m-1)}.$$
It remains to bound $a_2(K)$ in terms of $a_2(K_1\# K_2)$.
Recall that there is a sequence of at most $\frac{g}{3m}$ positive-to-negative crossing changes
that transforms $K$ into $K_1 \# K_2$.
At each step, $a_2$ changes by the linking number of a certain two-component link~$L$, as explained above. In the first step, $L$ is positive, so $\lk(L) \geq 0$;
in the second step, $L$ has at most one negative crossing, so still $\lk(L) \geq 0$;
more generally, in the $(2k-1)$-st and $2k$-th step, $\lk(L) \geq 1 - k$.
This implies
\begin{align*}
a_2(K) & \geq a_2(K_1 \# K_2) - \Bigl(0 + 0 + 1 + 1 + \ldots + \Bigl\lfloor\frac{g}{6m}\Bigr\rfloor + \Bigl\lfloor\frac{g}{6m}\Bigr\rfloor\Bigr) \\
 & \geq a_2(K_1 \# K_2) - \Bigl\lfloor\frac{g}{6m}\Bigr\rfloor\Bigl(\Bigl\lfloor\frac{g}{6m}\Bigr\rfloor+1\Bigr)
   \geq a_2(K_1 \# K_2) - \frac{g^2}{18m^2}.
\end{align*}
We conclude 
\[
a_2(K) \geq \frac{(g-\frac{g}{3m})^2}{6(m-1)} - \frac{g^2}{18m^2} \geq \frac{g^2}{6m}.
\]
%The first inequality makes use of the monotonicity of the Casson invariant on closures of positive braids under the operation of changing positive crossings to negative ones.
The last inequality is a consequence of the inequality
\[
3m^2\Bigl(1-\frac{1}{3m}\Bigr)^2 \geq (m-1) + 3m(m-1),
\]
which is true for all positive~$m$.
%$$m\left(1-\frac{1}{2m}\right)^2 \geq m-1.
%\eqno\qed$$\renewcommand{\qed}{}
\end{proof}

\begin{proof}[Proof of \Cref{quasi}]
As explained in~\cite{NNU}, any crossing change in a diagram of a knot $K$ with $n$ crossings can be realised as a sequence of at most $n$ $\Delta$-moves. Therefore, the $\Delta$\nobreakdash-unknotting number is bounded above, as follows:
$$u_\Delta(K) \leq \frac{n^2}{2}.$$
Now let $K$ be the closure of a reduced positive braid $\beta$ of length~$n$ in $B_m^+$. Here reduced simply means that every positive generator $\sigma_i$ appears at least twice in $\beta$, which implies $n \geq 2(m-1)$. By invoking Bennequin's genus formula one more time, $2g(K)=n-m+1$, we conclude
$n=2g(K)+m-1 \leq 2g(K)+\frac{n}{2}$, thus $n \leq 4g(K)$, in turn
$u_\Delta(K) \leq 8g(K)^2$. 
The first inequality in \Cref{casson} then implies
$$u_\Delta(K) \leq 48m a_2(K).$$

Now let $K_1,K_2 \in \widehat{B_m^+} \cup \widehat{B_m^-}$. By the inequality $u_\Delta(K) \geq a_2(K)$, the additivity of the Casson invariant $a_2$, as well as its invariance under the mirror image of knots, we obtain the following chain of inequalities:
\begin{align*}
    u_\Delta(K_1\#K_2) \geq a_2(K_1\#K_2)=a_2(K_1)+a_2(K_2) \geq \frac{u_\Delta(K_1)}{48m}+\frac{u_\Delta(K_2)}{48m}.
\end{align*}

\end{proof}

We thus established quasi-additivity of $u_{\Delta}$ for positive and negative braids with a fixed number of strands.
As mentioned in the introduction, we do not know whether %the unknotting numbers $u$ and
$u_\Delta$ is quasi-additive for all knots.
Other questions that have been settled for the ordinary unknotting number~$u$
still remain open for $u_\Delta$.
For example, the authors are not aware whether $u_\Delta(K)=1$ implies that $K$ is prime, an implication known to be true for $u$~\cite{Sch}.

\bigskip
\noindent
Mathematisches Institut, Universit\"at Bern, Sidlerstrasse 5, 3012 Bern, Switzerland

\smallskip
\noindent
Department of Mathematics, ETH Z\"urich, R\"amistrasse 101, 8092 Z\"urich, Switzerland 

\bigskip
\noindent
\texttt{\href{mailto:sebastian.baader@unibe.ch}{sebastian.baader@unibe.ch}}

\smallskip
\noindent
\texttt{\href{mailto:lukas.lewark@math.ethz.ch}{lukas.lewark@math.ethz.ch}}


\begin{thebibliography}{10}

\bibitem{B}
     D.~Bennequin: \emph{Entrelacements et \'{e}quations de Pfaff}, Third Schnepfenried geometry conference, Vol.~1 (Schnepfenried, 1982), 87--161.
Ast\'{e}risque, 107--108.

\bibitem{BH}
     M.~Brittenham, S.~Hermiller: \emph{Unknotting number is not additive under connected sum}, arXiv:2506.24088.

%\bibitem{FW}
%     J.~Franks, R.~F.~Williams: \emph{Braids and the Jones polynomial}, Trans. Amer.
%Math. Soc.~303 (1987), no.~1, 97--108.

\bibitem{H}
     K.~Habiro: \emph{Claspers and finite type invariants of links},
Geom. Topol.~4 (2000), 1--83.

\bibitem{I}
     T.~Ito: \emph{A note on HOMFLY polynomial of positive braid links}, Internat. J.~Math. 33 (2022), no.~4.

%\bibitem{K}
%     T.~Kawamura: \emph{On unknotting numbers and four-dimensional clasp numbers of links}, Proc. Amer. Math. Soc.~130 (2002), no.~1, 243--252.

\bibitem{L}
     M.~Lackenby: \emph{The crossing number of composite knots}, J. Topol.~2 (2009), no.~4, 747--768.

\bibitem{MN}
     H.~Murakami, Y.~Nakanishi: \emph{On a certain move generating link-homology},
Math. Ann.~284 (1989), no.~1, 75--89.

\bibitem{NNU}
     K.~Nakamura, Y.~Nakanishi, Y.~Uchida: \emph{Delta-unknotting number for knots},
J.~Knot Theory Ramifications~7 (1998), no.~5, 639--650.

\bibitem{O}
     M.~Okada: \emph{Delta-unknotting operation and the second coefficient of the Conway polynomial}, J.~Math. Soc. Japan~42 (1990), no.~4, 713--717.

\bibitem{PV}
     M.~Polyak, O.~Viro: \emph{On the Casson knot invariant}, Knots in Hellas '98, Vol.~3 (Delphi), J.~Knot Theory Ramifications~10 (2001), no.~5, 711--738.

\bibitem{R}
     L.~Rudolph: \emph{Braided surfaces and Seifert ribbons for closed braids}, Comment.~Math.~Helv.~58, 1--37 (1983). 

\bibitem{Sch}
     M.~Scharlemann: \emph{Unknotting number one knots are prime},
Invent. Math.~82 (1985), no.~1, 37--55.

\bibitem{St}
     A.~Stoimenow: \emph{Positive knots, closed braids and the Jones polynomial},
Ann. Sc. Norm. Super. Pisa Cl. Sci. (5) 2 (2003), no.~2, 237--285.

\bibitem{VB}
     J.~M.~Van Buskirk: \emph{Positive knots have positive Conway polynomials}, Knot theory and manifolds (Vancouver, B.C., 1983), 146--159. Lecture Notes in Math., 1144
Springer-Verlag, Berlin, 1985.


\end{thebibliography}
\end{document}